\documentclass[12pt,twoside]{article}
\usepackage{amsmath,amssymb,amsthm,mathrsfs,color,times,textcomp,sectsty,verbatim}
\usepackage[colorlinks=true]{hyperref}
\hypersetup{urlcolor=blue, citecolor=red, linkcolor=blue}
 \def \no{\nonumber}

 \allowdisplaybreaks
\begin{document}
\renewcommand{\proofname}{\bf Proof}
\let\oldsection\section
\renewcommand\section{\setcounter{equation}{0}\oldsection}
\renewcommand\thesection{\arabic{section}}
\renewcommand\theequation{\thesection.\arabic{equation}}
\newtheorem{claim}{\indent Claim}[section]
\newtheorem{theorem}{\indent Theorem}[section]
\newtheorem{lemma}{\indent Lemma}[section]
\newtheorem{proposition}{\indent Proposition}[section]
\newtheorem{definition}{\indent Definition}[section]
\newtheorem{remark}{\indent Remark}[section]
\newtheorem{corollary}{\indent Corollary}[section]
\newtheorem{example}{\indent Example}[section]
\title{\Large \bf
A nonlocal $\mathbf Q$-curvature flow on a class of closed manifolds of  dimension $\mathbf{n \geq 5}$}

\author{Xuezhang  Chen\thanks{ X. Chen: xuezhangchen@nju.edu.cn}\\
 \small
$^\ast$Department of Mathematics \& IMS, Nanjing University, Nanjing
210093, P. R. China
}

\date{}
\maketitle
\begin{abstract}
In this paper, we employ a nonlocal $Q$-curvature flow inspired by Gursky-Malchiodi's work \cite{gur_mal} to solve the prescribed $Q$-curvature problem on a class of closed manifolds: For $n \geq 5$, let $(M^n,g_0)$ be a smooth closed manifold, which is not conformally diffeomorphic to the standard sphere, satisfying either Gursky-Malchiodi's semipositivity hypotheses: scalar curvature $R_{g_0}>0$ and $Q_{g_0} \geq 0$ not identically zero or Hang-Yang's: Yamabe constant $Y(g_0)>0$, Paneitz-Sobolev constant $q(g_0)>0$ and  $Q_{g_0} \geq 0$ not identically zero. Let $f$ be a smooth positive function on $M^n$ and $x_0$ be some maximum point of $f$. Suppose either (a) $n=5,6,7$ or $(M^n,g_0)$ is locally conformally flat; or (b) $n \geq 8$, Weyl tensor at $x_0$ is nonzero. In addition, assume all partial derivatives of $f$ vanish at $x_0$ up to order $n-4$, then there exists a conformal metric $g$ of $g_0$ with its $Q$-curvature $Q_g$ equal to $f$. This result generalizes Escobar-Schoen's work [Invent. Math. 1986] on prescribed scalar curvature problem on any locally conformally flat manifolds of positive scalar curvature.

 {{\bf $\mathbf{2010}$ MSC:} Primary 53C44,58J05,53C21,45K05;~~ Secondary 35J30,35B40,35B50.}

\end{abstract}

\tableofcontents

\section{Introduction}

\indent \indent Let $(M^n,g)$ be a smooth Riemannian manifold  of dimension
$n \geq 3$, and $R_g, \text{Ric}_g$ be the scalar curvature,
Ricci curvature of the metric $g$, respectively. The following
conformally covariant operator of order four and $Q$-curvature on $(M^n,g)$ are first
discovered by S. Paneitz \cite{paneitz} in dimension four. Subsequently, in dimension $n\geq 3$ and $n \neq 4$, they are  recognized by T. Branson \cite{branson} and Fefferman-Graham
\cite{feffgra}, concretely,
\begin{eqnarray}
P_g&=&\Delta_g^2-\delta [(a_n R_g g+b_n \text{Ric}_g)d]+{n-4 \over 2}Q_g, \label{P_g}\label{Paneitz-Branson}\\
&& \hbox{~~with~~}a_n={(n-2)^2+4 \over 2(n-1)(n-2)}, b_n=-{4 \over n-2},\no
\end{eqnarray}
where $\delta$ is the divergent operator, $d$ is the exterior
derivative operator and the $Q$-curvature $Q_g$ of metric $g$ is defined by
\begin{eqnarray*}
Q_g=-{1 \over 2(n-1)}\Delta_g R_g+{n^3-4n^2+16n-16 \over 8(n-1)^2(n-2)^2} R_g^2-{2 \over (n-2)^2}|\text{Ric}_g|^2.
\end{eqnarray*}
Under  conformal change $\tilde{g}=u^{4 \over n-4}g$ for $n \neq 4$, there holds
\begin{equation}\label{conformal-invariant}
P_g(\varphi u)=u^{n+4 \over n-4}P_{\tilde{g}}(\varphi)
\end{equation}
for all $\varphi \in C^\infty(M^n)$. 

\indent Analogous to the Yamabe problem and the
prescribed scalar curvature problem, the prescribed
$Q$-curvature problem on closed manifolds involving the fourth order
conformally covariant operator can be proposed as follows: Give a
smooth function $f$ defined on a closed manifold $(M^n,g_0)$ with $n \geq 5$, can one find a
conformal metric $g=u^{4 \over n-4}g_0, u>0$ such that the
$Q$-curvature of the metric $g$ is equal to $f$? The prescribed
$Q$-curvature problem on $(M^n,g_0)$ with $n \geq 5$ is reduced to the
solvability of the equation
\begin{equation}\label{prescribed_Q-curvature}
P_{g_0}( u)={n-4 \over 2}f u^{n+4 \over n-4} \hbox{~~and~~} u>0 \hbox{~~on~~} M^n,
\end{equation}
where $P_{g_0}$ is fourth-order conformally covariant operator on $(M^n,g_0)$ defined by \eqref{Paneitz-Branson}. Very recently, some remarkable progresses have been made in the constant $Q$-curvature problem on closed manifolds after Chang-Yang's pioneering work \cite{cy}. For the existence of conformal metric to the constant $Q$-curvature problem, in which $f$ is a positive constant, under the assumptions that (a) $Q_{g_0} \geq 0$ and positive somewhere; (b) the scalar curvature $R_{g_0} \geq 0$, an affirmative answer is recently provided by Gursky-Malchiodi in \cite{gur_mal} adopting a nonlocal flow approach. Especially, a strong maximum principle for Paneitz operator was first established in their paper.  Later, Hang-Yang in their recent works \cite{hy1,hy2,hy3} relaxed the positivity of scalar curvature $R_{g_0}$ in Gursky-Malchiodi's assumptions to the Yamabe constant $Y(g_0)>0$ and the Paneitz-Sobolev constant $q(g_0)>0$, and used variational method to solve the constant $Q$-curvature problem on a class of closed manifolds of dimensions other than four. For more background on the $Q$-curvature problem, one may refer to \cite{cy,dhl,gur_mal,jlx} and the references therein.

For more than one decade, conformal geometric flows have played important roles in prescribed curvature problems in conformal geometry and have become very powerful tools in such problems comparing with the classical variational methods. In $2003$, S. Brendle initiated a negative gradient flow
approach in \cite{bre2} to deal with the prescribed $Q$-curvature
problem with critical exponential exponent on a closed Riemmanian manifold $(M^n,g_0)$. Not long after Brendle's work, M. Struwe adopted
this approach in \cite{str} to study the Nirenberg problem, namely
the prescribed Gauss curvature problem. Then  Malchiodi and Struwe
applied this approach to the prescribed $Q$-curvature problem on
$S^4$ in \cite{mal_str}, where they made some further developments
in Morse theory part. Recently, the author and X. Xu in \cite{chxu} adopted this method to the perturbation result for the prescribed scalar
curvature problem on $S^n$ with $n \geq 3$. However, to
the author's knowledge, in the present
stage, the presence of only one simple bubble in blow-up analysis is crucial to all the above mentioned works. The closely related topics to the above conformal geometric flows are the Yamabe flow \cite{bre3,bre1} and the fractional Yamabe flow \cite{jx}. Unfortunately, there still  exist some technical difficulties to show the long time existence of \underline{positive} solutions of such conformal geometric flows of higher order, even though such a strong maximum principle (e.g. for Paneitz operator in some special class of closed manifolds) has been established. After Baird-Fardoun-Regbaoui's work \cite{bfr} and Gursky-Malchiodi's \cite{gur_mal} emerged, the nonlocal flow comes into play and has shown its power in such prescribed curvature problems.

We are now in position to state our main theorem.
\begin{theorem}\label{main_Thm}
Let $(M^n, g_0)$ be a smooth closed manifold of dimension $n\geq 5$ and not conformally diffeomorphic to the standard sphere $S^n$. Let $f$ be a smooth positive function on $M^n$ and $x_0$ be a maximum point of $f$. In the case of either (a) $n=5,6,7$ or $(M^n,g_0)$ is locally conformally flat, or (b) $n \geq 8$, Weyl tensor at $x_0$  is nonzero, suppose either $R_{g_0}>0, Q_{g_0}\geq 0$ not identically zero or the Yamabe constant $Y(g_0)>0$, the Paneitz-Sobolev constant $q(g_0)>0$, $Q_{g_0}\geq 0$ not identically zero.  In addition, assume
\begin{equation}
 \nabla_{g_0}^l f(x_0)=0 \hbox{~~for~~} 1 \leq l \leq n-4.\label{vanishing_order}
\end{equation}
Then there exists a conformal metric $g$ of $g_0$ with its $Q_g$-curvature equal to $f$.
\end{theorem}

\begin{remark}
Theorem \ref{main_Thm} is a generalization of Theorem 2.1 in Escobar-Schoen's work \cite{es} on prescribed scalar curvature problem for any locally conformally flat manifold of dimension at least three, which is not conformally diffeomorphic to the standard sphere.
\end{remark}

\begin{remark}\label{leftcase}
One case left open in Theorem \ref{main_Thm} is that $(M^n,g_0)$ with $n \geq 8$ is not locally conformally flat and Weyl tensor is zero at any maximum point of $f$, due to the lack of Positive Mass Theorem for Paneitz operator, which is an obstruction in the construction of initial data of the nonlocal $Q$-curvature flow in this case.
\end{remark}

The vanishing order condition \eqref{vanishing_order} on $f$ at some maximum point is used to construct some positive initial data  of the nonlocal $Q$-curvature flow satisfying either semi-positivity hypotheses \eqref{gm_pos} or \eqref{hy_pos} below together with some restrictions on energy bounds. In dimension five, the condition \eqref{vanishing_order} is automatically satisfied.

In section \ref{sect2}, we introduce a nonlocal $Q$-curvature flow which is a negative gradient flow of $E_f[u]$ in some suitable Hilbert space, and a detailed proof for short time existence of the flow is available. In section \ref{sect3}, we show the positivity of $u(t,\cdot)$ for any time $t \geq 0$, as well as some elementary estimates involving $E_f[u]$ and $\alpha(t)$ etc. In section \ref{sect4}, the global existence of the nonlocal flow for some special class of initial data is presented. In section \ref{sect5}, we show the asymptotic convergence of $\int_{M^n}u_t P_{g_0}u_td\mu_{g_0}$ and the positivities of $Q$-curvature, as well as of the scalar curvature under hypotheses \eqref{gm_pos} for time $t \geq 0$.  Reminiscing about Aubin and Schoen's dichotomy on the Yamabe problem, in either (a) $n=5,6,7$ or $(M^n,g_0)$ is locally conformally flat, or (b) $n \geq 8$, Weyl tensor at some maximum point of $f$ is nonzero cases, in section \ref{sect6}, we construct  initial data satisfying either \eqref{gm_pos} or \eqref{hy_pos}, as well as some restrictions on initial energy bounds. Finally, in section \ref{sect7}, we establish sequential convergence of the nonlocal flow meanwhile completing the proof of Theorem \ref{main_Thm}. \\
{\bf Acknowledgments:} 
The author is supported through NSFC (No.11201223), A Foundation for the Author of National Excellent Doctoral Dissertation of China (No. 201417)  and Program for New Century Excellent Talents in University (NCET-13-0271). He is grateful to Professor Xingwang Xu for his invitation to the Institute for Mathematical Sciences at National University of Singapore, where parts of the work were carried out. He would like to thank Professor Yanyan Li and Dr. Jingang Xiong for so many helpful discussions and so much precious advice during this project, especially the latter part of the current proof of Lemma  \ref{local_existence} suggested by Professor Yanyan Li.

\section{A non-local $\mathbf{Q}$-curvature flow}\label{sect2}
\indent \indent Throughout this paper, suppose $(M^n, g_0)$ is a smooth closed manifold of dimension $n \geq 5$, satisfying either Gursky-Malchiodi's semipositivity hypotheses:
\begin{equation}\label{gm_pos}
\hbox{scalar curvature~~} R_{g_0} \geq 0, Q_{g_0} \geq 0 \hbox{~~and~~} Q_{g_0}>0 \hbox{~~somewhere};
\end{equation}
or Hang-Yang's:
\begin{equation}\label{hy_pos}
\hbox{Yamabe constant~~} Y(g_0)>0, q(g_0)>0 \hbox{~~and~~} Q_{g_0} \geq 0 \hbox{~~is not identically zero.}
\end{equation}

Under the above assumption \eqref{gm_pos}, a strong maximum principle for Paneitz operator $P_{g_0}$ in \cite{gur_mal} plays an important role in the existence of constant $Q$-curvature problem on closed manifolds. Define a Paneitz-Sobolev constant by
$$q(g_0)=q(M^n, g_0):= \inf\left\{\frac{\int_{M^n}w P_{g_0}wd\mu_{g_0}}
{(\int_{M^n}w^{2n \over n-4}d\mu_{g_0})^{n-4 \over n}}; w \in
H^2(M^n,g_0)\setminus \{0\}\right\}$$
which is independent of the selection of the metric in the conformal class of $g_0$. The Paneitz-Sobolev constant enjoys an analogous property of Yamabe constant proved by T. Aubin \cite{aubin}: 
\begin{lemma}\label{Paneitz_constant}
On a closed Riemannian manifold $(M^n,g_0)$ of dimension $n \geq 8$, suppose there exists $p \in M^n$ such that the Weyl tensor $W_{g_0}(p) \neq 0$, then $q(g_0)<q(S^n)$. 
\end{lemma}

One may refer to the author's unpublished note \cite{chen} or the process of the proof in \cite{er} Theorem 1.1 for a detailed proof of Lemma \ref{Paneitz_constant}. 

Later, Hang and Yang (cf. \cite{hy2,hy3}) relaxed the above hypotheses \eqref{gm_pos} to weaker ones \eqref{hy_pos}, such a strong maximum principle still holds true. Both conditions \eqref{gm_pos} and \eqref{hy_pos} derive that $\hbox{ker}P_{g_0}=0$. However, up to now, it is not clear of the coercivity of Paneitz operator or $q(g_0)>0$ can follow from $Y(g_0)>0, Q_{g_0} \geq 0$ not identically zero whether or not.
 
Notice that under the Gursky-Malchiodi's hypotheses \eqref{gm_pos}, it follows from Proposition 2.3 in \cite{gur_mal} that $P_{g_0}$ is coercive and $q(g_0)$ is positive. Then, under either the hypotheses \eqref{gm_pos} or \eqref{hy_pos}, we claim that the norm 
$$\langle v, P_{g_0}v \rangle_{g_0}^{1 \over 2}=\Big(\int_{M^n}v P_{g_0}v d\mu_{g_0}\Big)^{1 \over 2}, \hbox{~~for~~} v \in H^2(M^n,g_0)$$
and $\|v\|_{H^2(M^n,g_0)}$ are equivalent. Indeed, by the positivity of Paneitz-Sobolev constant $q(g_0)$ and H\"{o}lder's inequality, there holds
$$\int_{M^n} v^2 d\mu_{g_0} \leq \hbox{vol}(M^n,g_0)^{4 \over n} \Big(\int_{M^n} v^{2n \over n-4}d\mu_{g_0}\Big)^{n-4 \over n}\leq \frac{\hbox{vol}(M^n,g_0)^{4 \over n}}{q(g_0)} \int_{M^n}vP_{g_0}v d\mu_{g_0}.$$
From the interpolation Sobolev inequality, given $0<\epsilon<1$, one has
\begin{eqnarray*}
\int_{M^n}|\Delta_{g_0}v|^2d\mu_{g_0}&=&\int_{M^n}vP_{g_0}v d\mu_{g_0}+\Big(\int_{M^n}|\Delta_{g_0}v|^2 d\mu_{g_0}-\int_{M^n}vP_{g_0}v d\mu_{g_0}\Big)\\
&\leq& \int_{M^n}vP_{g_0}v d\mu_{g_0}+\epsilon\int_{M^n}|\Delta_{g_0}v|^2d\mu_{g_0}+C_\epsilon \int_{M^n}v^2 d\mu_{g_0}.
\end{eqnarray*}
Putting the above facts together, one has
\begin{eqnarray*}
\int_{M^n}|\Delta_{g_0}v|^2d\mu_{g_0}&\leq& \frac{1}{1-\epsilon}\int_{M^n}vP_{g_0}v d\mu_{g_0}+C_\epsilon \int_{M^n}v^2 d\mu_{g_0}\\
&\leq& C\int_{M^n}vP_{g_0}v d\mu_{g_0},
\end{eqnarray*}
which yields
$$\|v\|^2_{H^2(M^n,g_0)}\leq C \int_{M^n}vP_{g_0}v d\mu_{g_0}.$$
On the other hand, it is easy to verify the inverted direction of the above inequality.

Let $f$ be a smooth positive function defined on $M^n$. Motivated by Baird-Fardoun-Regbaoui \cite{bfr} and Gursky-Malchiodi \cite{gur_mal}, we extend their ideas to introduce a nonlocal $Q$-curvature flow:
\begin{eqnarray}\label{gm_Q_flow_orig}
\left\{
\begin{array}{ll}
\frac{\partial u}{\partial t}=-u+\frac{n-4}{2} P_{g_0}^{-1}\big( \alpha f |u|^{\tfrac{n+4}{n-4}}\big) \hbox{~~for~~} (x,t) \in M^n \times [0,T);\\
u(0,x)=u_0\in C^\infty(M^n);
\end{array}
\right.
\end{eqnarray}
coupled with the constraint function of time $t$:
\begin{equation}\label{constraint_factor}
\alpha(t)=\frac{2}{n-4}\frac{\int_{M^n}uP_{g_0}u d\mu_{g_0}}{\int_{M^n}fu^{2n \over n-4}d \mu_{g_0}}.
\end{equation}

\begin{remark}
From \cite{gur_mal} Proposition 2.4 or \cite{hy2} Lemma 3.2, there exists a Green's function $G_{g_0}(p,\cdot)$ of the Paneitz operator $P_{g_0}$ with pole at $p \in M^n$, such that $G_{g_0}(p,\cdot)>0$ in $M^n \setminus \{p\}$. Thus, given a function $f \in C^\infty(M^n)$, the operator of $P_{g_0}^{-1}(f)$ can be interpreted as the convolution between the Green's function $G_{g_0}(p,\cdot)$ of $P_{g_0}$ and the function $f$, in other words,
$$P_{g_0}^{-1}(f)(p)=\int_{M^n}G_{g_0}(p,x)f(x)d\mu_{g_0}(x).$$
\end{remark}

\begin{remark}\label{rm1}
Under conformal change of metrics $g=u^{\tfrac{4}{n-4}}g_0$, using the conformal covariance \eqref{conformal-invariant} of the Paneitz operator $P_g$, we find that its inverse $P_g^{-1}$ is conformally covariant:
$$P_g^{-1}=u^{-1}P_{g_0}^{-1}\big(u^{\tfrac{n+4}{n-4}}\cdot\big),$$
due to the simple fact that $\forall \psi \in C^\infty(M^n)$,
$$u^{-1}P_{g_0}^{-1}\big(u^{\tfrac{n+4}{n-4}}P_g\psi\big)=u^{-1}P_{g_0}^{-1}\big(P_{g_0}(u\psi)\big)=\psi.$$
\end{remark}

\subsection{Short time existence}
\indent \indent This subsection is devoted to the study of the short time existence to the flow problem \eqref{gm_Q_flow_orig}.
\begin{lemma}\label{local_existence}
There exists a unique solution in $C^0([0,T]; C^{4,\lambda}(M^n))$ to the flow problem \eqref{gm_Q_flow_orig} for any $0<\lambda<1$ and $0<T\leq \infty$.
\end{lemma}
\begin{proof}
It is sufficient to show the short time existence of the following modified flow problem
\begin{eqnarray}\label{gm_Q_flow_modified}
\left\{
\begin{array}{ll}
\frac{\partial \tilde{u}}{\partial s}=-\tilde{u}+\frac{n-4}{2}P_{g_0}^{-1}\big(f |\tilde{u}|^{\tfrac{n+4}{n-4}}\big) \hbox{~~for~~} (x,s) \in M^n \times [0,T);\\
\tilde{u}(s=0,x)=u_0\in C^\infty(M^n).
\end{array}
\right.
\end{eqnarray}
Indeed, the modified flow problem \eqref{gm_Q_flow_modified} differs from the original one \eqref{gm_Q_flow_orig} by scalings in time and functions. Set
$$s(t)=\int_0^t \mu(\tau)d\tau \hbox{~~and~~} u(t,x)=e^{s(t)-t}\tilde{u}(x,s(t))$$
where
$$\mu(t)=\frac{2}{n-4}\frac{\int_{M^n} \tilde{u}P_{g_0}\tilde{u}d\mu_{g_0}}{\int_{M^n}f \tilde{u}^{2n \over n-4}d\mu_{g_0}}.$$
By using the expressions of $\alpha(t)$ and $\mu(t)$ to get
$$\alpha(t)=\mu(t) e^{-\tfrac{8}{n-4}(s(t)-t)},$$
we conclude that if the short time existence of the above modified flow problem \eqref{gm_Q_flow_modified} is established, so is the flow problem \eqref{gm_Q_flow_orig}.

Next we manage to establish the short time existence to \eqref{gm_Q_flow_modified} by the contraction mapping principle. For simplicity we still use $u$ instead of $\tilde u$ and time variable $t$ instead of $s$. Let $X=C^0([0,T]; C^{4,\lambda}(M^n))$ for some $\lambda \in (0,1)$. For any fixed $\delta>0$, define
$$X_\delta:=\{v \in X; v(0,x)=u_0, \|v-u_0\|_X \leq \delta\}$$
where $u_0 \in C^\infty(M^n)$, and put the distance on $X_\delta$ by
$$\rho(v,w):=\|v-w\|_X \hbox{~~for~~} v,w \in X_\delta.$$
It is not hard to verify that the space $(X_\delta,\rho)$ is a complete metric space. Define a map $L: X_\delta \to X$ by
$$L(u)(t,x)=u_0(x)-\int_0^t u(\tau,x)d\tau+\frac{n-4}{2}\int_0^t P_{g_0}^{-1}(f|u|^{n+4 \over n-4})(\tau,x)d\tau.$$
Recall that, from the Schauder estimates of elliptic equations, there holds
$$\|P_{g_0}^{-1}w\|_{C^{4,\lambda}(M^n)}\leq C \|w\|_{C^{\lambda}(M^n)},$$
where $C$ depends only on $n, \lambda$. Then one has
\begin{eqnarray*}
\|P_{g_0}^{-1}(f|u|^{\frac{n+4}{n-4}})(t)\|_{C^{4,\lambda}(M^n)}&\leq& C \|(f|u|^{\frac{n+4}{n-4}})(t)\|_{C^\lambda(M^n)}\\
&\leq& C \|u(t)\|_{C^0(M^n)}^{\frac{8}{n-4}}\|u(t)\|_{C^\lambda(M^n)},
\end{eqnarray*}
and then
\begin{eqnarray*}
&&\|L(u)(\cdot, t)-u_0(\cdot)\|_{C^{4,\lambda}(M^n)}\\
&\leq& T\|u(t)\|_{X}+C\int_0^t \|u(\tau)\|_{C^0(M^n)}^{\frac{8}{n-4}}\|u(\tau)\|_{C^\lambda(M^n)}d\tau\\
&\leq& CT(\|u\|_X+\|u\|_X^{\frac{n+4}{n-4}})\\
&\leq& CT(1+(\delta+\|u_0\|_{C^{4,\lambda}(M^n)})^{\frac{n+4}{n-4}}).
\end{eqnarray*}
Thus we have 
$$\|L(u)-u_0\|_X \leq C(f,n,\delta,\lambda) T.$$
By choosing $T>0$ sufficiently small, it is not hard to verify that the map $L$ is a contraction on $(X_\delta,\rho)$. Then, by the contraction mapping theorem, there exists a unique fixed point of $L$. Thus, the local well-possedness of the modified flow problem \eqref{gm_Q_flow_modified} is established. 
\end{proof}

\section{Positivity of $\mathbf{u(t,x)}$ and energy estimates}\label{sect3}

\indent \indent We first need to show that the positivity of $u$ is preserved along the flow. From now on, we impose some restrictions on initial data, that is $u_0 \in C^\infty_\ast$,where
$$C_\ast^\infty=\left\{w \in
C^\infty(M^n);w>0, P_{g_0}w\geq 0\right \}.$$
 It is easy to know $C_\ast^\infty \neq \emptyset$ since $1 \in C_\ast^\infty$ in view of the fact $P_{g_0}1=\tfrac{n-4}{2}Q_{g_0}\geq 0$ due to Gursky-Malchiodi's semi-positivity hypotheses \eqref{gm_pos} or Hang-Yang's \eqref{hy_pos}. 
\begin{lemma}\label{positivity_pres}
Let $u$ be a $C^4$-solution to the nonlocal flow equation $\eqref{gm_Q_flow_orig}$  with $u(0,x)=u_0(x) \in C^\infty_\ast$, then for all $0 \leq t \leq T$,  there hold
$$u(x,t)>0 \hbox{~~and~~} u \in C_\ast^\infty.$$
\end{lemma}
\begin{proof}
A direct computation yields
\begin{eqnarray*}
\frac{\partial }{\partial t}P_{g_0}(u)&=&P_{g_0}(u_t)\\
&=&-P_{g_0}u+\frac{n-4}{2}\alpha f |u|^{n+4 \over n-4}\\
&\geq& -P_{g_0}u.
\end{eqnarray*}
Then, we have
$$P_{g_0}u(t,x) \geq e^{-t}P_{g_0}(u_0)\geq 0,$$
in view of $u_0 \in C^\infty_\ast$. Thus, under either Gursky-Malchiodi's \eqref{gm_pos} or Hang-Yang's \eqref{hy_pos}, the strong maximum principle for $P_{g_0}$ (cf. \cite{gur_mal} Theorem 2.2 or \cite{hy2} Proposition 3.1, respectively) gives
$$\frac{n-4}{2}Q_g u^{n+4 \over n-4}=P_{g_0}u(t,x)\geq 0,~~u(t,x)>0 \hbox{~~for all~~} (x,t) \in M^n \times [0,T].$$
The above fact also implies $u\in C_\ast^\infty$. This completes the proof.
\end{proof}

Consequently, the flow problem \eqref{gm_Q_flow_orig} turns to
\begin{eqnarray}\label{gm_Q_flow}
\left\{
\begin{array}{ll}
\frac{\partial u}{\partial t}=-u+\frac{n-4}{2}P_{g_0}^{-1}\big(\alpha f u^{\tfrac{n+4}{n-4}}\big);\\
u(0,x)=u_0 \in C^\infty_\ast;
\end{array}
\right.
\end{eqnarray}
where $\alpha(t)$ is given in \eqref{constraint_factor}. For brevity, let
$$\varphi=-u+\tfrac{n-4}{2}P_{g_0}^{-1}(\alpha f u^{n+4 \over n-4}),$$
and then
\begin{equation}\label{heat_P_g_0_u}
\frac{\partial}{\partial t}P_{g_0}u=-P_{g_0}u+\frac{n-4}{2}\alpha f u^{n+4 \over n-4}=P_{g_0}\varphi.
\end{equation}

From now on, denote by $g(t)=u(t)^{4 \over n-2}g_0$ the flow metric and $d\mu_g=u(t)^{2n \over n-4}d\mu_{g_0}$ the volume form of the flow metric, then $Q$-curvature equation gives
\begin{equation}\label{Q_curvature_eqn}
P_{g_0} u=\frac{n-4}{2}Q u^{n+4 \over n-4} \hbox{~~on~~} M^n,
\end{equation}
where $Q=Q_g$ is the $Q$-curvature of the flow metric $g(t)$. Define the energy functionals
\begin{eqnarray*}
 E[u]&=&\frac{n-4}{2}\int_{M^n} Q_{g} d\mu_{g}=\int_{M \sp n}u P_{g_0}( u) d \mu_{g_0}\\
 &=&\int_{M^n}\Big[(\Delta_{g_0}u)^2+a_nR_{g_0} |\nabla u|_{g_0}^2+b_n \hbox{Ric}_{g_0}(\nabla u, \nabla u)+{n-4 \over 2}Q_{g_0}u^2\Big]d\mu_{g_0}
 \end{eqnarray*}
 and
 $$E_f[u]={E[u]\over \big(\int_{M^n}f u^{2n \over n-4}d\mu_{g_0}\big)^{(n-4)/n}}.$$
By \eqref{constraint_factor} and \eqref{Q_curvature_eqn}, we have
$$E[u(t)]=\frac{n-4}{2}\int_{M^n}Q d\mu_g$$
and
\begin{equation}\label{alpha_new}
0=\int_{M^n}(\alpha f-Q)d\mu_g=\frac{2}{n-4}\int_{M^n}uP_{g_0}\varphi d\mu_{g_0}.
\end{equation}

 Along this flow, the energy $E[u(t)]$ is preserved for all time $t \geq 0$.
\begin{lemma}\label{cons_energy}
Along the nonlocal flow \eqref{gm_Q_flow}, the energy $E[u(t)]$ is conserved for any time $t \geq 0$.
\end{lemma}
\begin{proof} By the flow equation \eqref{gm_Q_flow} and \eqref{constraint_factor}, we obtain
\begin{eqnarray*}
\frac{d}{dt}\int_{M^n}uP_{g_0}u d\mu_{g_0}&=&2\int_{M^n}u P_{g_0}(u_t)d\mu_{g_0}\\
&=&2\int_{M^n}uP_{g_0}\Big(-u+\tfrac{n-4}{2}P_{g_0}^{-1}(\alpha f u^{\tfrac{n+4}{n-4}})\Big)d\mu_{g_0}\\
&=&-2 \Big[\int_{M^n}uP_{g_0}u d\mu_{g_0}-\tfrac{n-4}{2}\alpha(t) \int_{M^n}f u^{2n \over n-4}d\mu_{g_0}\Big]\\
&=&0,
\end{eqnarray*}
which implies the desired assertion.
\end{proof}

From Remark \ref{rm1}, the flow equation \eqref{gm_Q_flow} is equivalent to
\begin{equation}\label{gm_Q_gradient_flow}
u_t=\frac{n-4}{2}P_g^{-1}(\alpha f-Q_g)u,
\end{equation}
that is,
$$\frac{\partial}{\partial t}g=2P_g^{-1}(\alpha f-Q_g) g.$$
Since $P_g$ is self-adjoint and positive, we can define $H^{2}(M^n,g)$ inner product by
$$\langle \eta,\zeta \rangle_g=\int_{M^n}\eta P_g \zeta d\mu_g,$$
which induces the $H^2$-norm $\|\cdot \|_{H^2(M^n,g)}$. In this sense, the nonlocal $Q$-curvature flow \eqref{gm_Q_flow} is a negative gradient flow of $E_f[u]$ in the Hilbert space $H^2(M^n,g)$. Now we pause for a while to give some explanations for the nonlocal $Q$-curvature flow. Up to a positive constant, regard $E_f[u]$ as a functional of the metric $g$:
$$\mathcal{Q}_f (g)=\frac{\int_M Q_g d\mu_g}{\Big(\int_M f d\mu_g\Big)^{n-4 \over n}}$$
Set
$$g_\epsilon=\phi_\epsilon g$$
satisfying
$$\phi_\epsilon\Big|_{\epsilon=0}=1 \hbox{~~and~~} \frac{d \phi_\epsilon}{d\epsilon}\Big|_{\epsilon=0}=\phi \in C^\infty(M).$$
Notice that $P_g$ is invertible under hypotheses \eqref{gm_pos} or \eqref{hy_pos}, we obtain
\begin{equation*}
\begin{split}
\mathcal{Q}'_f (g)[\phi]=&\frac{d}{d\epsilon}\Big|_{\epsilon=0}\frac{\int_M Q_{\phi_\epsilon g} d\mu_{\phi_\epsilon g}}{\Big(\int_M f \phi_\epsilon^{n \over 2}d\mu_g\Big)^{n-4 \over n}}\\
=&\frac{2}{n-4}\frac{d}{d\epsilon}\Big|_{\epsilon=0}\frac{\int_M \phi_\epsilon^{n-4 \over 4}P_g( \phi_\epsilon^{n-4 \over 4})d\mu_g}{\Big(\int_M f \phi_\epsilon^{n \over 2}d\mu_g\Big)^{n-4 \over n}}\\
=&\frac{\int_M \phi P_g(1) d\mu_g}{\Big(\int_M f d\mu_g\Big)^{n-4 \over n}}-\frac{n-4}{2}\frac{\int_M Q_g d\mu_g}{\Big(\int_M f d\mu_g\Big)^{{n-4 \over n}+1}\int_M f \phi d\mu_g}\\
=&\frac{n-4}{2}\Big(\int_M f \phi_\epsilon^{n \over 2}d\mu_g\Big)^{4-n \over n}\Big[\int_M \phi (Q_g-\alpha f)d\mu_g\Big]\\
=&\frac{n-4}{2}\Big(\int_M f \phi_\epsilon^{n \over 2}d\mu_g\Big)^{4-n \over n}\langle\phi, P_g^{-1} (Q_g-\alpha f)\rangle_g.
\end{split}
\end{equation*}

\begin{lemma}\label{energy_non-increasing}
Along the nonlocal flow \eqref{gm_Q_flow}, the energy $E_f[u(t)]$ is non-increasing for all time $t \geq 0$.
\end{lemma}
\begin{proof}
By \eqref{gm_Q_gradient_flow} and \eqref{constraint_factor}, we compute
\begin{eqnarray*}
\frac{d}{dt}E_f[u]&=&-(n-4)\Big(\int_{M^n}fu^{\tfrac{2n}{n-4}}d\mu_{g_0}\Big)^{\tfrac{4-n}{n}}\int_{M^n}\alpha fu^{-1}u_t d\mu_g\\
&=&-\frac{(n-4)^2}{2}\Big(\int_{M^n}fu^{\tfrac{2n}{n-4}}d\mu_{g_0}\Big)^{\tfrac{4-n}{n}}\int_{M^n}\alpha f P_g^{-1}(\alpha f-Q)d\mu_g.
\end{eqnarray*}
From \eqref{rm1}, \eqref{Q_curvature_eqn} and \eqref{constraint_factor}, as well as from \eqref{alpha_new}, we notice that
\begin{eqnarray*}
\int_{M^n} Q P_g^{-1}(\alpha f-Q)d\mu_g&=&\frac{2}{n-4}\int_{M^n}P_{g_0}(u) uP_g^{-1}(\alpha f-Q)d\mu_{g_0}\\
&=& \frac{2}{n-4}\int_{M^n}P_{g_0}(u)P_{g_0}^{-1}\Big((\alpha f-Q)u^{\tfrac{n+4}{n-4}}\Big)d\mu_{g_0}\\
&=&\frac{2}{n-4}\int_{M^n}(\alpha f-Q)d\mu_g=0.
\end{eqnarray*}
Thus, we obtain
\begin{eqnarray*}
\frac{d}{dt}E_f[u]&=&-\frac{(n-4)^2}{2}\Big(\int_{M^n}fu^{\tfrac{2n}{n-4}}d\mu_{g_0}\Big)^{\tfrac{4-n}{n}}\int_{M^n}(\alpha f-Q) P_g^{-1}(\alpha f-Q)d\mu_g\\
&=&-\frac{(n-4)^2}{2}\Big(\int_{M^n}fu^{\tfrac{2n}{n-4}}d\mu_{g_0}\Big)^{\tfrac{4-n}{n}}\|P_g^{-1}(\alpha f-Q)\|_{H^2(M^n,g)}^2 \leq 0.
\end{eqnarray*}
This completes the proof.
\end{proof}

\subsection{Some elementary estimates}

\begin{lemma}\label{vol_bdabove}
The conformal volume of the flow metric is uniformly bounded above, that is, for all time $t \geq 0$, there exists a positive constant $C_1$ depending on $n, q(g_0)$ and initial energy $E[u_0]$ such that
$$\int_{M^n} u^{2n \over n-4}d\mu_{g_0} \leq C_1.$$
Moreover, there exists a positive constant $C_2$ depending on $\max_{M^n}f, n, q(g_0)$ and initial energy $E[u_0]$, such that
$$\int_{M^n} f u^{2n \over n-4} d\mu_{g_0} \leq C_2.$$
\end{lemma}
\begin{proof}
By the definition of the Paneitz-Sobolev constant $q(g_0)$ and Lemma \ref{cons_energy}, one obtains
$$q(g_0)\Big(\int_{M^n} u^{2n \over n-4} d\mu_{g_0}\Big)^{n-4 \over n} \leq E[u(t)]=E[u_0].$$
Thus, it yields
$$\int_{M^n}u^{2n \over n-4}d\mu_{g_0} \leq  \Big(\frac{E[u_0]}{q(g_0)}\Big)^{n \over n-4} := C_1.$$
Moreover, we get
$$\int_{M^n}f u^{2n \over n-4}d\mu_{g_0} \leq (\max_{M^n}f) \int_{M^n}u^{2n \over n-4}d\mu_{g_0}\leq (\max_{M^n}f) C_1 := C_2.$$
This concludes the proof.
\end{proof}

Indeed, we manage to show that $\alpha(t)$ is uniformly bounded below and above, as well as is the conformal volume $\int_{M^n}u(t)^{\frac{2n}{n-4}}d\mu_{g_0}$.
\begin{lemma}\label{bd_alpha}
 There exist two uniform positive  constants $C_3=C_3(n,\max_{M^n}f, E[u_0])$ and $C_4=C_4(n, E[u_0])$ such that
$$0<C_3 \leq \alpha(t) \leq C_4.$$
\end{lemma}
\begin{proof}
By the expression \eqref{constraint_factor} of $\alpha(t)$ and the conservation of energy $E[u(t)]$ in view of Lemma \ref{cons_energy} , we obtain 
$$\alpha(t) \geq \frac{E[u_0]}{C_2} := C_3,$$
where $C_2$ is the positive constant given in Lemma \ref{vol_bdabove}. On the other hand, from \eqref{alpha_t} one asserts that $\alpha_t \leq 0$, which implies $\alpha(t) \leq \alpha(0) := C_4.$ This concludes the proof.
\end{proof}

\begin{lemma} \label{F_2}
For any fixed time $T>0$, there exists a uniform constant $C$ depending on $n, q(g_0), \max_{M^n}f$ and initial energy $E[u_0]$ such that 
$$\Big(\frac{n-4}{2}\Big)^2\int_0^T\int_{M^n}(\alpha f-Q) P_g^{-1}(\alpha f-Q)d\mu_g dt=\int_0^T \int_{M^n} \varphi P_{g_0}\varphi d\mu_{g_0}dt\leq C.$$
\end{lemma}
\begin{proof}
From the flow equation \eqref{gm_Q_flow} and Lemma \ref{cons_energy}, one has
\begin{eqnarray*}
&&\frac{d}{dt}\int_{M^n}\alpha f u^{2n \over n-4}d\mu_{g_0}\\
&=&\frac{2n}{n-4}\int_{M^n}\alpha f u^{n+4 \over n-4}u_t d\mu_{g_0}+\alpha_t \int_{M^n}fu^{2n \over n-4} d\mu_{g_0}\\
&=&\frac{2n}{n-4}\int_{M^n}\alpha f u^{n+4 \over n-4}\Big[-u+\tfrac{n-4}{2}P_{g_0}^{-1}(\alpha f u^{n+4 \over n-4})\Big]d\mu_{g_0}+\frac{2}{n-4}\frac{\alpha_t}{\alpha}\int_{M^n}u P_{g_0}u d\mu_{g_0}\\
&=&\frac{2n}{ n-4}\int_{M^n}\Big[\tfrac{n-4}{2}\alpha f u^{n+4 \over n-4}P_{g_0}^{-1}(\alpha fu^{n+4 \over n-4})-\alpha f u^{2n \over n-4}\Big]d\mu_{g_0}+\frac{2}{n-4}\frac{\alpha_t}{\alpha}E[u_0].
\end{eqnarray*}
On the other hand, by the $Q$-curvature equation \eqref{Q_curvature_eqn} and the expression \eqref{alpha_new} of $\alpha(t)$, we have
\begin{eqnarray*}
&&\int_{M^n}\varphi P_{g_0} \varphi d\mu_{g_0}\\
&=&\int_{M^n}\Big[-u+P_{g_0}^{-1}(\tfrac{n-4}{2}\alpha f u^{n+4 \over n-4})\Big]P_{g_0}\Big[-u+P_{g_0}^{-1}(\tfrac{n-4}{2}\alpha f u^{n+4 \over n-4})\Big]d\mu_{g_0}\\
&=&\int_{M^n}\Big[uP_{g_0}u-\tfrac{n-4}{2}\alpha f u^{2n \over n-4}\Big]d\mu_{g_0}\\
&&+\frac{n-4}{2}\int_{M^n}\Big[\tfrac{n-4}{2}\alpha f u^{n+4 \over n-4} P_{g_0}^{-1}(\alpha f u^{n+4 \over n-4})-\alpha f u^{2n \over n-4}\Big]d\mu_{g_0}\\
&=&\frac{n-4}{2}\int_{M^n}\Big[\tfrac{n-4}{2}\alpha f u^{n+4 \over n-4} P_{g_0}^{-1}(\alpha f u^{n+4 \over n-4})-\alpha f u^{2n \over n-4}\Big]d\mu_{g_0}.
\end{eqnarray*}
Therefore, by the definition of $\alpha(t)$ , we conclude that
\begin{equation}\label{alpha_t}
0=\frac{d}{dt}\int_{M^n}\alpha f u^{2n \over n-4}d\mu_{g_0}=\frac{2}{n-4}\Big(\frac{2n}{n-4}\int_{M^n}\varphi P_{g_0}\varphi d\mu_{g_0}+\frac{\alpha_t}{\alpha}E[u_0] \Big).
\end{equation}
Integrating the above equation over $[0,T]$ to obtain
\begin{equation}\label{intest_F_2}
\int_0^T \int_{M^n} \varphi P_{g_0}\varphi d\mu_{g_0}dt=\frac{n-4}{2n }E[u_0](\log \alpha(0)-\log \alpha(T)).
\end{equation}
By Lemma \ref{vol_bdabove}, it yields 
$$\log \alpha(T) \geq \log \Big(\frac{2}{n-4}\frac{E[u_0]}{C_2}\Big)$$ 
for any fixed $T>0$. Thus, from \eqref{intest_F_2} we conclude that
$$\int_0^T \int_{M^n} \varphi P_{g_0} \varphi d\mu_{g_0} dt \leq C.$$
This completes the proof.
\end{proof}

As a byproduct of Lemma \ref{F_2}, provided that the nonlocal flow globally exists for all time, we obtain 
\begin{equation}\label{int_F_2_bd}
\int_0^\infty \int_{M^n} \varphi P_{g_0} \varphi d\mu_{g_0} dt<\infty.
\end{equation}
In particular, there exists a sequence $\{t_j\}_{j=1}^{\infty}$ with $t_j \rightarrow \infty,$
 such that
 \begin{equation}\label{pointwise}
 \int_{M \sp n}\varphi(t_j) P_{g_0}(\varphi(t_j))d
\mu_{g_0} \rightarrow 0,\text{as } j \rightarrow \infty.
\end{equation}

\section{Global existence}\label{sect4}

\begin{lemma}\label{global_exist}
Given any initial data $u_0 \in C^\infty_\ast$, there exists a smooth solution to the nonlocal $Q$-curvature flow problem \eqref{gm_Q_flow}-\eqref{constraint_factor} for all time $t \geq 0$. Moreover, for any fixed time $T>0$, there exist two positive constants $C$ and $C'$ depending on $n,T,f$ and initial data $u_0$ such that
\begin{equation}\label{bd_u}
\|u(t)\|_{C^0(M^n)}\leq C e^{C' t} \hbox{~~for~~} 0 \leq t \leq T.
\end{equation}
\end{lemma}
\begin{proof}
Let $\theta>1$. From the proof of Lemma \ref{positivity_pres}, we obtain that 
\begin{equation}\label{two_positivity}
u(t,x)>0 \hbox{~~and~~} P_{g_0}u(t,x)>0
\end{equation}
as long as the flow exists. Along with this fact, by the equation \eqref{heat_P_g_0_u} and Lemma \ref{bd_alpha}, we have
\begin{eqnarray}
\frac{d}{dt}\int_{M^n}\big(P_{g_0}u\big)^\theta d\mu_{g_0}&=&\theta \int_{M^n}\big(P_{g_0}u\big)^{\theta-1} P_{g_0}u_t d\mu_{g_0}\no\\
&=&\theta \int_{M^n}\big(P_{g_0}u\big)^{\theta-1}\Big(-P_{g_0}u+\frac{n-4}{2}\alpha f u^{n+4 \over n-4}\Big)d\mu_{g_0}\no\\
&\leq&-\theta \int_{M^n}\big(P_{g_0}u\big)^\theta d\mu_{g_0}+\overline{C}_\theta \int_{M^n} \big(P_{g_0}u\big)^{\theta-1}u^{n+4 \over n-4} d\mu_{g_0}.\label{est_P_S^n_u}
\end{eqnarray}
Using H\"{o}lder's inequality, we estimate
$$\int_{M^n} \big(P_{g_0}u\big)^{\theta-1}u^{n+4 \over n-4} d\mu_{g_0}\leq \Big(\int_{M^n}\big(P_{g_0} u\big)^\theta d\mu_{g_0}\Big)^{\theta-1 \over \theta} \Big(\int_{M^n}u^{(n+4)\theta \over n-4}d\mu_{g_0}\Big)^{1 \over \theta}.$$

First choose
$$1<\theta <\frac{n}{4}.$$
By H\"{o}lder's inequality, one gets
\begin{eqnarray*}
\Big(\int_{M^n}u^{(n+4)\theta \over n-4}d\mu_{g_0}\Big)^{1 \over \theta}&=&\Big(\int_{M^n}u^\theta u^{8 \theta \over n-4} d\mu_{g_0}\Big)^{1 \over \theta}\\
&\leq&\Big(\int_{M^n} u^{n\theta \over n-4\theta}d\mu_{g_0}\Big)^{n-4\theta \over n \theta}\Big(\int_{M^n}u^{2n \over n-4}d \mu_{g_0}\Big)^{4 \over n}.
\end{eqnarray*}
By the Sobolev embedding $W^{4,\theta}(M^n,g_0) \hookrightarrow L^{n\theta \over n-4\theta}(M^n,g_0)$ and the basic fact that
$$\|u\|_{W^{4,\theta}(M^n,g_0)}\approx \|P_{g_0}u\|_{L^\theta(M^n,g_0)},$$
we obtain
$$\Big(\int_{M^n} u^{n\theta \over n-4\theta}d\mu_{g_0}\Big)^{n-4\theta \over \theta n}\leq C_\theta \Big(\int_{M^n}\big(P_{g_0}u\big)^\theta d\mu_{g_0}\Big)^{1 \over \theta}.$$
Thus, together with Lemma \ref{vol_bdabove}, one obtains
$$\int_{M^n} \big(P_{g_0}u\big)^{\theta-1}u^{n+4 \over n-4} d\mu_{g_0}\leq C_\theta \int_{M^n}\big(P_{g_0}u\big)^\theta d\mu_{g_0}.$$
Hence, substituting these above facts into \eqref{est_P_S^n_u} and using \eqref{two_positivity} to show
$$\frac{d}{dt}\int_{M^n}\big(P_{g_0}u\big)^\theta d\mu_{g_0}\leq C_\theta \int_{M^n}\big(P_{g_0}u\big)^\theta d\mu_{g_0}.$$
Integrating the above over $(0,t)$ to get
$$\int_{M^n}\big(P_{g_0}u\big)^\theta d\mu_{g_0}\leq C_\theta e^{C'_\theta t}$$
for all $0 \leq t \leq T$. Again using the Sobolev embedding theorem, one gets
$$\|u\|_{L^{n\theta \over n-4\theta}(M^n,g_0)}\leq C_\theta e^{C'_\theta t}.$$
By choosing $\theta$ sufficiently tending to $\tfrac{n}{4}$, we establish that for any $p>1$, there holds
\begin{equation}\label{L^p_u}
\|u(t)\|_{L^p(M^n,g_0)}\leq C_p e^{C'_p t}.
\end{equation}

Next fix $\theta=p>\tfrac{n}{4}$ in \eqref{est_P_S^n_u}, applying \eqref{L^p_u}, H\"{o}lder's and Young's inequalities to estimate
\begin{eqnarray*}
&&\int_{M^n}\big(P_{g_0}u\big)^{p-1} u^{n+4 \over n-4} d\mu_{g_0}\\
&\leq&\Big(\int_{M^n}\big(P_{g_0}u\big)^p d\mu_{g_0}\Big)^{p-1 \over p}\Big(\int_{M^n}u^{(n+4)p \over n-4}d\mu_{g_0}\Big)^{1 \over p}\\
&\leq&\Big(C_p e^{C'_p t}\Big)^{1 \over p}\Big(\int_{M^n}\big(P_{g_0}u\big)^p d\mu_{g_0}\Big)^{p-1 \over p}\\
&\leq & C_p e^{C'_p t}\Big(\int_{M^n}\big(P_{g_0}u\big)^p d\mu_{g_0}\Big)^{p-1 \over p}\\
&\leq& \frac{p}{\overline{C}_p}\int_{M^n}\big(P_{g_0}u\big)^p d\mu_{g_0}+C_p e^{C'_p t}.
\end{eqnarray*}
Substituting the above back to \eqref{est_P_S^n_u} to show
$$\frac{d}{dt}\int_{M^n}\big(P_{g_0}u\big)^p d\mu_{g_0} \leq C_p e^{C'_p t}.$$
Along with the Sobolev embedding theorem, integrating the above over $(0,t)$ to get
$$\|u(t)\|_{C^\lambda(M^n)} \leq C_p \|P_{g_0}u(t)\|_{L^p(M^n,g_0)}\leq C_p e^{C_p t},$$
for all $0 \leq t \leq T$, where $\lambda=4-\tfrac{n}{p} \in (0,1)$. Obviously, the above estimate implies \eqref{bd_u}. Going back to the flow equation \eqref{gm_Q_flow}, through \eqref{heat_P_g_0_u} and \eqref{bd_u}, we conclude that $C^{4,\lambda}$-norm of $u(t)$ has at most exponential growth in any finite time interval. Therefore, the phenomenon of finite time blow-up is excluded and then the global existence of the flow equation \eqref{gm_Q_flow}-\eqref{constraint_factor} is established.
\end{proof}

\section{Asymptotic behaviors}\label{sect5}

\indent \indent In this section, we establish asymptotic convergence of $\int_{M^n}\varphi P_{g_0}\varphi d\mu_{g_0}$, as well as the positivity of $Q$-curvature of the flow metric.
\begin{lemma}\label{pos_Q_R}
Let $u_0 \in C_\ast^\infty$, then there holds $Q(t,x)>0$ for all $(x,t) \in M^n \times (0,\infty)$. Moreover, under hypotheses \eqref{gm_pos}, there holds $R_g(t,x)>0$ for all $(x,t) \in M^n \times [0,\infty)$.
\end{lemma}
\begin{proof}
By the positivities of $u(t),\alpha(t)$ and the definition of $C_\ast^\infty$, integrating \eqref{heat_P_g_0_u} over $(0,t)$ to show
$$\frac{n-4}{2}Qu^{n+4 \over n-4}=P_{g_0}u(t)=P_{g_0}u_0+\frac{n-4}{2}\int_{0}^t e^{\tau-t}\alpha(\tau)fu^{n+4 \over n-4}(\tau)d\tau>0,$$
which implies the first assertion. Next, define
$$t_\ast:=\sup \{t\in [0,\infty); \min_{M^n}R(\cdot,t)=0\}.$$
Notice that the set on the right side is nonempty under Gursky-Malchiodi's hypotheses \eqref{gm_pos} . We claim that $t_\ast=+\infty$. Suppose $t_\ast<\infty$. Since $R_{g_0}>0$ together with the positivity of $Q$, for $t \in (0,t_\ast]$ we have
$$R_g(t) \geq 0 \hbox{~~and~~} -{1 \over 2(n-1)}\Delta_g R_g+{n^3-4n^2+16n-16 \over 8(n-1)^2(n-2)^2} R_g^2>0.$$
By applying strong maximum principle to $R_g$ at time $t=t_\ast$,  one obtains that $R_g(t_\ast)>0$ or $R_g(t_\ast) \equiv 0$. However, in the latter case, it yields that $Q(t_\ast)=-{2 \over (n-2)^2}|\text{Ric}_g(t_\ast)|^2 \leq 0$, which contracts the positivity of $Q(t_\ast)$. Then $R_g(t_\ast)>0$, which also contradicts the definition of $t_\ast$. From this, the second assertion follows. 
\end{proof}

From the positivity of the scalar curvature $R_g$, we obtain some lower bounds of $u(t)$.
\begin{lemma}
Under the hypotheses \eqref{gm_pos}, there exists a positive constant $C$ depending on $g_0$ and $R_{g_0}$ such that
\begin{eqnarray*}
 \int_{M^n} u(t)^{n-2 \over n-4} d\mu_{g_0} &\leq& C \Big(\inf_{M^n}u(t)\Big)^{n-2 \over n-4} \hbox{~~and~~}\\
 \Big(\int_{M^n}u(t)^{2n \over n-4}d\mu_{g_0}\Big)^{n-4 \over n-2} &\leq& C \inf_{M^n}u(t) \Big(\sup_{M^n}u(t)\Big)^{n+2 \over n-2},
 \end{eqnarray*}
 for all $t \geq 0$.
 \end{lemma}
\begin{proof}
Under the hypotheses \eqref{gm_pos}, by Lemma \ref{pos_Q_R} and the scalar curvature equation of the flow metric $g=u(t)^{4 \over n-4}g_0$:
$$-\frac{4 (n-1)}{n-2}\Delta_{g_0} u^{n-2 \over n-4}+R_{g_0}u^{n-2 \over n-4}=R_g u^{n+2 \over n-4}>0,$$
these two assertions follow from Lemma A.2 and Corollary A.3 in \cite{bre3}, respectively.
\end{proof}

Let
$$F_2(t)=\int_{M^n}\varphi P_{g_0}\varphi d\mu_{g_0}.$$

Notice that $\alpha(t)$ is non-increasing with respect to time $t \geq 0$ along the flow. More precisely, we obtain
\begin{lemma}\label{bd_alpha_t}
There exists a positive uniform constant $C=C(n,\max_{M^n}f,E[u_0])$ such that
$$\alpha_t =-\frac{2n}{n-4} \frac{\alpha}{E[u_0]}F_2(t) \leq -C F_2(t)\leq 0\hbox{~~~~for~~} t \geq 0.$$
\end{lemma}
\begin{proof}
By the equation \eqref{alpha_t} and Lemma \ref{bd_alpha}, one has
$$\alpha_t=-\frac{2n}{n-4} \frac{\alpha}{E[u_0]}F_2(t) \leq -C F_2(t)\leq 0$$
as desired.
\end{proof}

\begin{lemma}\label{asym_F_2}
There holds
$$\lim\limits_{t \to \infty}F_2(t)=0.$$
\end{lemma}
\begin{proof}
By \eqref{gm_Q_flow} and \eqref{heat_P_g_0_u}, we have
\begin{eqnarray}
\frac{1}{2}\frac{d}{d t}F_2(t)&=&\int_{M^n}\varphi P_{g_0} \varphi_t d\mu_{g_0}\no\\
&=&\int_{M^n} \varphi \frac{\partial}{\partial t}\Big(-P_{g_0} u+\frac{n-4}{2}\alpha f u^{n+4 \over n-4}\Big)d\mu_{g_0}\no\\
&=&\int_{M^n}\varphi\Big(-P_{g_0}\varphi+\frac{n-4}{2}\alpha_t f u^{n+4 \over n-4}+\frac{n+4}{2}\alpha f u^{8 \over n-4}\varphi\Big)d\mu_{g_0}.\label{dt_F_2}
\end{eqnarray}
By Lemmas \ref{vol_bdabove} and \ref{bd_alpha_t}, using H\"{o}lder's inequality and Sobolev embedding, we estimate the second integral by
\begin{eqnarray*}
\Big|\int_{M^n}\alpha_t f u^{n+4 \over n-4}\varphi d\mu_{g_0}\Big|&\leq& C F_2(t)\Big(\int_{M^n}\varphi^{2n \over n-4}d\mu_{g_0}\Big)^{n-4 \over 2n}\Big(\int_{M^n} u^{2n \over n-4}d \mu_{g_0}\Big)^{n+4 \over 2n}\\
&\leq& C F_2^{3 \over 2}(t).
\end{eqnarray*}
By Lemmas \ref{vol_bdabove} and \ref{bd_alpha}, employing H\"{o}lder's inequality and Sobolev embedding to bound
\begin{eqnarray*}
\Big|\int_{M^n}\alpha f u^{8 \over n-4} \varphi^2 d\mu_{g_0}\Big|&\leq& C\Big(\int_{M^n}u^{2n \over n-4}d\mu_{g_0}\Big)^{4 \over n}\Big(\int_{M^n}\varphi^{2n \over n-4}d\mu_{g_0}\Big)^{n-4 \over n}\\
&\leq& C F_2(t).
\end{eqnarray*}
Thus, by \eqref{dt_F_2}, we obtain that
\begin{equation}\label{diff_ineq1}
\frac{d}{dt} F_2(t) \leq CF_2(t)(1+F_2(t)^{1 \over 2}).
\end{equation}
By \eqref{int_F_2_bd}, there exists a sequence $\{t_j\}$ with $t_j \to \infty$ as $j \to \infty$, such that
$$\lim_{t \to \infty}F_2(t_j)=0.$$
Set
$$H(t)=\int_0^{F_2(t)}\frac{ds}{1+s^{1 \over 2}}=2F_2(t)^{\tfrac{1}{2}}-2\log\big(1+F_2(t)^{1 \over 2}\big),$$
from which one has $\lim\limits_{j \to \infty}H(t_j)=0$. Integrating \eqref{diff_ineq1} over $(t_j,t)$ for any $t \geq t_j$ to show
$$H(t) \leq H(t_j)+C \int_{t_j}^t F_2(\tau)d\tau.$$
Letting $j \to \infty$ and by \eqref{pointwise}, we have $\lim_{t \to \infty}H(t)=0$.
Then we assert that there exists some uniform constant $C_0>0$ such that  
$$F_2(t) \leq C_0 \hbox{~~for all time~~} t \geq 0.$$ 
To show this by negation, otherwise there exists a sequence $\{t_k\}$ with $t_k \to \infty$ as $k \to \infty$ such that
$F_2(t_k)>1$ for all $k \in \mathbb{N}$. However,
$$H(t_k)=\int_0^{F_2(t_k)} \frac{ds}{1+s^{1 \over 2}}ds\geq \int_0^1\frac{ds}{1+s^{1 \over 2}}ds>0$$
which contradicts $\lim_{t \to \infty}F_2(t)=0$. Also it is easy to show $F_2(t) \leq C_0 H(t)$ for all time $t \geq 0$. Therefore, we conclude that 
$$\lim\limits_{t \to \infty}F_2(t) \leq C_0 \lim\limits_{t \to \infty}H(t)=0.$$
This completes the proof.
\end{proof}

\section{Construction of initial data}\label{sect6}

\indent \indent The objective of  this section is to construct some positive initial data satisfying either semi-positivity hypotheses \eqref{gm_pos} or \eqref{hy_pos}, as well as some restrictions on initial energy bounds. Such initial data are very crucial in establishing sequential convergence of the nonlocal flow in the later part.

Let $x_0 \in M^n$ be a maximum point of $f$. It follows from Lee-Parker \cite{lp} that  there exists a conformal metric $\tilde{g}=\varphi^{4 \over n-4} g_0$ with conformal normal  coordinates around $x_0$. If $g_0$ is locally conformally flat, we choose $\tilde{g}=\varphi^{4 \over n-4} g_0$ flat near $x_0$. Since some estimates and computations needed in this paper have been available in \cite{gur_mal}, we mostly adopt the same notation used in \cite{gur_mal} for simplicity. Set
$$b_n=n(n-4)(n^2-4), ~~c_n=\frac{1}{2(n-2)(n-4)\omega_{n-1}},$$
and
$$u_\epsilon(x)=\frac{\chi_\delta(x)}{(\epsilon^2+d_{\tilde g}(x,x_0)^2)^{n-4 \over 2}}.$$
where $\omega_{n-1}=\hbox{vol}(S^{n-1},g_{S^{n-1}})$ and $\chi_\delta(x)$ is a nonnegative smooth cut off function supported in $B_{2\delta}(x_0)$ satisfying $\chi_\delta(x)=1$ in $B_{\delta}(x_0)$ and $\chi_\delta(x)=0$ outside of $B_{2\delta}(x_0)$. For brevity, we use $\nabla$ instead of $\nabla_{\tilde g}$.

A direct computation yields the following asymptotics
\begin{eqnarray*}
&&\int_{M^n}u_\epsilon^{2n \over n-4}d\mu_{\tilde g}=O(\epsilon^{-n}), \int_{M^n} u_\epsilon^{n+4 \over n-4}d\mu_{\tilde g}=O(\epsilon^{-4}), \int_{M^n}u_\epsilon^{8 \over n-4}d\mu_{\tilde g}=O(\epsilon^{n-8}),\\
&&b_n \epsilon^4 \left(\int_{M^n}u_\epsilon^{2n \over n-4} d\mu_{\tilde g}\right)^{4 \over n}=q(S^n)(1+O(\epsilon^n)).
\end{eqnarray*}

Since $P_{\tilde g}$ is invertible under hypotheses \eqref{hy_pos} or by Proposition $2.3$ in \cite{gur_mal} under hypotheses \eqref{gm_pos}, we consider $\hat{u}_{\epsilon}$ satisfying
\begin{equation}\label{hat_u_epsilon}
P_{\tilde{g}}\hat{u}_{\epsilon}=\frac{b_n\epsilon^4 \chi_\delta(x)}{(\epsilon^2+|x|^2)^{n+4 \over 2}}.
\end{equation}

Due to the property of Paneitz-Sobolev constant stated in Lemma \ref{Paneitz_constant} and Positive Mass Theorem (cf. \cite{gur_mal} Proposition 2.5 and \cite{hr} Theorem 1.1), for technical reasons, we divide the construction of initial data  into two cases.

\subsection{$\mathbf{(M^n,g_0): n \geq 8}$ and not locally conformally flat at $\mathbf{x_0}$}

\begin{lemma}\label{initial_hd_nlcf}
Let $(M^n, g_0)$ be a closed manifold of dimension $n \geq 8$ and there exists  a maximum point $x_0$ of $f$ such that $W_{g_0}(x_0)\neq 0$.  Suppose either the semipositivity hypotheses \eqref{gm_pos} or \eqref{hy_pos} holds. Moreover, assume
\begin{equation}
 \nabla_{g_0}^l f(x_0)=0 \hbox{~~for~~}2 \leq l \leq n-4,\label{higher-dim}
\end{equation} 
Then for sufficiently small $\epsilon>0$, there exist a positive function $u_{0\epsilon}$ and a positive constant $C_{x_0,f,n}$ such that
\begin{eqnarray*}
E_f[u_{0\epsilon}] &\leq& \frac{q(S^n)}{(\max_{M^n}f)^{n-4 \over n}}(1-C_{x_0,f} \epsilon^4 |\log \epsilon| |W_{g_0}(x_0)|^2) \hbox{~~if~~} n=8;\\
E_f[u_{0\epsilon}] &\leq& \frac{q(S^n)}{(\max_{M^n}f)^{n-4 \over n}}(1-C_{x_0,f,n} \epsilon^4  |W_{g_0}(x_0)|^2) \hbox{\qquad~ if~~} n\geq 9.
\end{eqnarray*}
Moreover, a  conformal metric $\bar g=u_{0 \epsilon}^{4 \over n-4} g_0$ enjoys the property that $Q_{\bar g}\geq 0$ and positive somewhere. In addition, under assumptions \eqref{gm_pos}, the scalar curvature $R_{\bar g}$ is positive.
\end{lemma}
\begin{proof} Define
$$\mathcal{F}_g[u]=\frac{\int_{M^n}uP_g u d\mu_g}{\Big(\int_{M^n}u^{2n \over n-4}d\mu_g\Big)^{n-4 \over n}}. $$

From \cite{chen} or \cite{er}, for sufficiently small $\epsilon>0$,  one has
\begin{eqnarray}
\mathcal{F}_{\tilde{g}}[u_\epsilon] &\leq& q(S^n)(1-O(\epsilon^4 |\log \epsilon| |W_{g_0}(x_0)|^2)) \hbox{~~if~~} n=8;\no\\
\mathcal{F}_{\tilde{g}}[u_\epsilon] &\leq& q(S^n)(1-O(\epsilon^4 |W_{g_0}(x_0)|^2)) \hbox{\qquad\quad if~~} n\geq9.\label{hd_energy_est}
\end{eqnarray}
Also, $\mathcal{F}_{\tilde{g}}[u_\epsilon]=\mathcal{F}_{g_0}[u_\epsilon \varphi]$ by conformal invariance of Paneitz operator. Let $v_\epsilon=\hat{u}_\epsilon-u_\epsilon$, Lemma 4.3 in \cite{gur_mal} gives
\begin{eqnarray}
|v_\epsilon|&\leq& C \log \frac{1}{\epsilon^2+|x|^2} \hbox{~~~if~~} n=8;\no\\
|v_\epsilon|&\leq& C (\epsilon^2+|x|^2)^{8-n \over 2} \hbox{~~if~~} n \geq 9.\label{est_v_epsilon}
\end{eqnarray}
The same computations in the proof of Lemma 4.4 in \cite{gur_mal} yield
$$\int_{M^n}\hat{u}_\epsilon P_{\tilde g}\hat{u}_\epsilon d\mu_{\tilde g}=\int_{M^n}u_\epsilon P_{\tilde g}u_\epsilon d\mu_{\tilde g}+2b_n \epsilon^4 \int_{M^n}u_\epsilon^{n+4 \over n-4}v_\epsilon d\mu_{\tilde g}+O(1)$$
and
$$\int_{M^n}u_\epsilon P_{\tilde g} u_\epsilon d\mu_{\tilde g}=b_n \epsilon^4(1+o_\epsilon(1))\int_{M^n}u_\epsilon^{2n \over n-4}d\mu_{\tilde g}.$$

One simple observation is that the covariant derivatives of $f$ with respect to metric $g_0$ at $x_0$ vanish up to order $m \geq 1$, then its covariant derivatives of $f$ with respect to conformal metric $g$ of $g_0$ at $x_0$ vanish up to the same order.
Then, from condition \eqref{higher-dim}, near $x_0$ there holds
$$f(x)=f(x_0)+\tfrac{1}{(n-3)!}\nabla_{i_1 \cdots i_{n-3}} f(x_0)x^{i_1}\cdots x^{i_{n-3}}+O(r^{n-2}).$$
In $B_\delta(x_0)$, from \eqref{est_v_epsilon} there holds $v_\epsilon \leq C u_\epsilon$, then
$$\Big||u_\epsilon+v_\epsilon|^{2n \over n-4}-u_\epsilon^{2n \over n-4}-\frac{2n}{n-4}u_\epsilon^{n+4 \over n-4}v_\epsilon\Big| \leq Cu_\epsilon^{8 \over n-4}v_\epsilon^2.$$
Putting these facts together, we obtain
\begin{eqnarray*}
\int_{M^n}f \hat{u}_\epsilon^{2n \over n-4}d\mu_{\tilde g}&=&\int_{M^n}f |u_\epsilon+v_\epsilon|^{2n \over n-4}d\mu_{\tilde g}\\
&=&\int_{B_\delta(x_0)}f\Big(u_\epsilon^{2n \over n-4}+\frac{2n}{n-4}u_\epsilon^{n+4 \over n-4}v_\epsilon+O(u_\epsilon^{8 \over n-4}v_\epsilon^2)\Big)d\mu_{\tilde g}+O(1)\\
&=&f(x_0)\int_{B_\delta(x_0)}\Big(u_\epsilon^{2n \over n-4}+\frac{2n}{n-4}u_\epsilon^{n+4 \over n-4}v_\epsilon+O(u_\epsilon^{8 \over n-4}v_\epsilon^2)\Big)d\mu_{\tilde g}+O(1)\\
&&+\int_{B_\delta}O(|x|^{n-3})\Big(u_\epsilon^{2n \over n-4}+\frac{2n}{n-4}u_\epsilon^{n+4 \over n-4}v_\epsilon+O(u_\epsilon^{8 \over n-4}v_\epsilon^2)\Big)d\mu_{\tilde g}.
\end{eqnarray*}
Now we set $u_{0\epsilon}=\varphi \hat{u}_\epsilon$. Combining the above estimates and the following asymptotics
\begin{eqnarray*}
&&\int_{M^n}u_\epsilon^{n+4 \over n-4} v_\epsilon d\mu_{\tilde g}
=\left\{\begin{array}{ll}
O(\epsilon^{-4}|\log \epsilon|) &\hbox{~~if~~} n=8;\\
O(\epsilon^{4-n}) &\hbox{~~if~~} n \geq 9;
\end{array}
\right.\\
&&\int_{M^n}u_\epsilon^{8 \over n-4} v_\epsilon^2 d\mu_{\tilde g}
=\left\{\begin{array}{ll}
O(|\log \epsilon|^3) &\hbox{~~if~~} n=8;\\
O(\epsilon^{8-n}) &\hbox{~~if~~} n \geq 9;
\end{array}
\right.\\
&& \int_{M^n}|x|^{n-3} u_\epsilon^{2n \over n-4}d\mu_{\tilde g}=O(\epsilon^{-3});\\
&&\int_{M^n}|x|^{n-3} u_\epsilon^{n+4 \over n-4}v_\epsilon d\mu_{\tilde g}
=\left\{\begin{array}{ll}
O(|\log \epsilon|^2) &\hbox{~~if~~} n=8;\\
O(1) &\hbox{~~if~~} n \geq 9;
\end{array}
\right.\\
&& \int_{M^n}|x|^{n-3} u_\epsilon^{8 \over n-4}v_\epsilon^2 d\mu_{\tilde g}=\left\{\begin{array}{ll}
O(|\log \epsilon|^4) &\hbox{~~if~~} n=8;\\
O(1) &\hbox{~~if~~} n \geq 9;
\end{array}
\right.\\
\end{eqnarray*}
we conclude that
\begin{eqnarray}
E_f[u_{0\epsilon}]&=&\frac{\int_{M^n}\hat{u}_\epsilon P_{\tilde g}\hat{u}_\epsilon d\mu_{\tilde g}}{\Big(\int_{M^n}f\hat{u}_\epsilon^{2n \over n-4}d\mu_{\tilde g}\Big)^{n-4 \over n}}\no\\
&=& \frac{\mathcal{F}_{\tilde g}[u_\epsilon]}{(\max_{M^n}f)^{n-4 \over n}} \frac{1+\tfrac{2 b_n \epsilon^4 \int_{M^n} u_\epsilon^{n+4 \over n-4}v_\epsilon d\mu_{\tilde g}}{\int_{M^n} u_\epsilon P_{\tilde g} u_\epsilon d\mu_{\tilde g}}+O_{\delta}(\epsilon^n)}{\Big(1+\tfrac{2n}{n-4}\tfrac{\int_{M^n} u_\epsilon^{n+4 \over n-4}v_\epsilon d\mu_{\tilde g}}{\int_{M^n} u_\epsilon^{2n \over n-4}d\mu_{\tilde g}}+ \tfrac{\int_{M^n} O(u_\epsilon^{8 \over n-4}v_\epsilon^2)d\mu_{\tilde g}}{\int_{M^n} u_\epsilon^{2n \over n-4}d\mu_{\tilde g}}+O_{\delta}(\epsilon^{n-3})\Big)^{n-4 \over n}}\no\\
&=&\frac{\mathcal{F}_{\tilde g}[u_\epsilon]}{(\max_{M^n}f)^{n-4 \over n}} \frac{1+\tfrac{2 \int_{M^n} u_\epsilon^{n+4 \over n-4}v_\epsilon d\mu_{\tilde g}}{(1+o_\epsilon(1))\int_{M^n} u_\epsilon^{2n \over n-4} d\mu_{\tilde g}}+O_{\delta}(\epsilon^n)}{\Big(1+\tfrac{2n}{n-4}\tfrac{\int_{M^n} u_\epsilon^{n+4 \over n-4}v_\epsilon d\mu_{\tilde g}}{\int_{M^n} u_\epsilon^{2n \over n-4}d\mu_{\tilde g}}+ \tfrac{\int_{M^n} O(u_\epsilon^{8 \over n-4}v_\epsilon^2)d\mu_{\tilde g}}{\int_{M^n} u_\epsilon^{2n \over n-4}d\mu_{\tilde g}}+O_{\delta}(\epsilon^{n-3})\Big)^{n-4 \over n}}\no\\
&=& \frac{\mathcal{F}_{\tilde g}[u_\epsilon]}{(\max_{M^n}f)^{n-4 \over n}}\Big(1+O_{\delta}(\epsilon^{\min\{8,n-3\}})\Big). \label{initial_data_est_nlcf_hd}
\end{eqnarray}
From \eqref{initial_data_est_nlcf_hd}, \eqref{hd_energy_est} together with the above estimates, we conclude the first assertion. 

By equation \eqref{hat_u_epsilon} and conformal invariance of Paneitz operator, one has
\begin{eqnarray*}
 P_{g_0}(u_{0 \epsilon})=P_{\tilde g}(\hat u_\epsilon)\varphi^{n+4 \over n-4}=\frac{b_n\epsilon^4 \chi_\delta(x)}{(\epsilon^2+|x|^2)^{n+4 \over 2}}\varphi^{n+4 \over n-4} \geq 0.
\end{eqnarray*}
Together with $Q_{g_0}\geq 0$ not identically zero, under either Gursky-Malchiodi's hypotheses \eqref{gm_pos} or Hang-Yang's \eqref{hy_pos}, the strong maximum principle for Paneitz operator (cf. \cite{gur_mal} Theorem 2.2 or \cite{hy2} Proposition 3.1, respectively) shows that $u_{0\epsilon}>0$ and then $\hat{u}_\epsilon>0$. Again by conformal invariance of Paneitz operator, the identity
$$\frac{n-4}{2} Q_{\bar g}=u_{0\epsilon}^{-\frac{n+4}{n-4}}P_{g_0}(u_{0\epsilon})=\frac{b_n\epsilon^4 \chi_\delta(x)}{(\epsilon^2+|x|^2)^{n+4 \over 2}}\hat{u}_\epsilon^{-\frac{n+4}{n-4}} ,$$
which implies that $Q_{\bar g} \geq 0$ is not identically zero.The positivity of $R_{\bar g}$ under hypotheses \eqref{gm_pos} follows from Lemma 4.6 in \cite{gur_mal}. This completes the proof.
\end{proof}

\subsection{$\mathbf{(M^n,g_0): n=5,6,7}$ or $\mathbf{n \geq 8}$ locally conformally flat}

\begin{lemma}\label{initial_ld_lcf}
Let $(M^n, g_0)$ be a closed manifold of dimensions $n=5,6,7$ or locally conformally flat. Suppose either the semi positivity hypotheses \eqref{gm_pos} or \eqref{hy_pos} holds. Let $f$ be a smooth positive function on $M^n$. Moreover, for $n \geq 5$, let $x_0$ denote a maximum point of $f$, assume
\begin{eqnarray}
 &\hbox{if}& n=6, ~~\nabla_{g_0}^2 f(x_0)=0;\label{six-dim}\\
 &\hbox{if}& n=7,~~\nabla_{g_0}^l f(x_0)=0 \hbox{~~for~~}2 \leq l \leq 3;\label{seven-dim}\\
 &\hbox{if}& n \geq 8 \hbox{~and~}g_0 \hbox{ is locally conformally flat},\nabla_{g_0}^l f(x_0)=0 \hbox{~~for~~} 2 \leq l \leq n-4. \label{lcf-higherdim}
\end{eqnarray}
Then for sufficiently small $\epsilon>0$, there exist a positive function $u_{0\epsilon}$ and a positive constant $C_{x_0,f,n}$ such that
$$E_f[u_{0\epsilon}] \leq \frac{q(S^n)}{(\max_{M^n}f)^{n-4 \over n}}-C_{x_0,f,n}\epsilon.$$
Moreover, a  conformal metric $\bar g=u_{0 \epsilon}^{4 \over n-4} g_0$ enjoys the property that $Q_{\bar g}\geq 0$ and positive somewhere. In addition, under assumptions \eqref{gm_pos}, the scalar curvature $R_{\bar g}$ is positive.
\end{lemma}
\begin{proof}
We set $u_{0\epsilon}=\varphi \hat{u}_\epsilon$. Let $G_{\tilde{g}}(x_0,\cdot)$ denote the Green's function with pole at $x_0$, from the Positive Mass Theorem for Paneitz operator in dimensions $n=5,6,7$ (cf. \cite{gur_mal} Proposition 2.5) or $(M^n,g_0)$ is locally conformally flat (cf. \cite{hr} Theorem 1.1), $G_{\tilde{g}}(x_0,\cdot)$ has the expansion of 
$$G_{\tilde g}(x_0,x)=c_n d_{\tilde g}(x_0,x)^{4-n}+\alpha_{x_0}+O(r),$$ 
where $r=d_{\tilde g}(x_0,x)$ and $\alpha_{x_0}>0$ is the mass. 
\begin{remark}\label{Rm_lcf_PMT}
In the latter case, just as pointed out by Humbert-Raulot in their paper \cite{hr}, under the assumptions of our Theorem \ref{main_Thm}, the condition ``$(M^n,g_0)$ is locally conformally flat" can not be weaken to ``$g_0$ is conformally flat around one (maximum) point $x_0$", since their proof of Positive Mass Theorem for Paneitz operator involves the standard Schoen-Yau's Positve Mass Theorem \cite{sy} for conformal Laplacian. We wonder whether some more restrictions on Weyl tensor at some maximum point of $f$ will be helpful to construct initial data in this case or not, for instance, just like the one in \cite{hv} for scalar curvature case.
\end{remark}
Let $\beta=\frac{\alpha_{x_0}}{c_n}$, then applying the computations on page 36 in \cite{gur_mal} together with conformal invariance of Paneitz operator to show
\begin{eqnarray*}
E[u_{0\epsilon}]&=&\int_{M^n}\hat{u}_\epsilon P_{\tilde g} \hat u_\epsilon d\mu_{\tilde g}\\
&=&b_n \epsilon^4\Big(\int_{M^n}u_\epsilon^{2n \over n-4}d\mu_{\tilde g}+\beta (1+o(1))\int_{M^n}u_\epsilon^{n+4 \over n-4}d\mu_{\tilde g}+O(1)\Big).
\end{eqnarray*}

It remains to estimate $\int_{M^n} f\hat{u}_\epsilon^{2n \over n-4}d\mu_{\tilde{g}}$. For some $0<\tilde{\delta}<< \delta$, a cut-off function $\chi_{\tilde \delta}(x)$ may similarly defined as $\chi_\delta(x)$ before. Set
$$\check{u}_\epsilon=\chi_{\tilde \delta}(u_\epsilon+\beta)+(1-\chi_{\tilde \delta})\bar G_{x_0} \hbox{~~with~~} \bar{G}_{x_0}=\frac{G_{\tilde g}(x_0,\cdot)}{c_n},$$
and rewrite $\hat u_\epsilon$ as
$$\hat{u}_\epsilon=\chi_{\tilde \delta}(u_\epsilon+\beta)+(1-\chi_{\tilde \delta})\bar G_{x_0}+\hat{u}_\epsilon-\check{u}_\epsilon.$$

Since in $B_{\tilde \delta}(x_0)$, $\beta$ is bounded by $u_\epsilon$, we have
\begin{eqnarray*}
&&\int_{B_{\tilde \delta}(x_0)}(u_\epsilon+\beta)^{2n \over n-4}d\mu_{\tilde g}\\
&=&\int_{M^n}u_\epsilon^{2n \over n-4}d\mu_{\tilde g}+\frac{2n}{n-4}\beta\int_{M^n}u_\epsilon^{n+4 \over n-4}d\mu_{\tilde g}+\beta^2 \int_{M^n}O(u_\epsilon^{8 \over n-4})d\mu_{\tilde g}+O(1),
\end{eqnarray*}

For this purpose, we estimate $\int_{M^n} f\hat{u}_\epsilon^{2n \over n-4}d\mu_{\tilde{g}}$ and also $E_f[u_{0\epsilon}]$ by dividing into two cases. 

(a) In dimensions $n=5,6,7$, from conditions \eqref{six-dim} and \eqref{seven-dim}, in a neighborhood of $x_0$ there holds
$$f(x)=f(x_0)+\tfrac{1}{(n-3)!}\nabla_{i_1\cdots i_{n-3}} f(x_0)x^{i_1}\cdots x^{i_{n-3}}+O(r^{n-2}).$$
Then, together with
$$|\hat{u}_\epsilon-\check{u}_\epsilon|=o_{\tilde \delta}(1)$$
in view of Lemma 5.3 in \cite{gur_mal}, we estimate
\begin{eqnarray*}
&&\int_{M^n}f\hat{u}_\epsilon^{2n \over n-4}d\mu_{\tilde g}\\
&=&\int_{B_{\tilde \delta}(x_0)}f(u_\epsilon+\beta)^{2n \over n-4}d\mu_{\tilde g}+O(1)\\
&=&\int_{B_{\tilde \delta(x_0)}}(f(x_0)+\tfrac{1}{(n-3)!}\nabla_{i_1\cdots i_{n-3}} f(x_0)x^{i_1}\cdots x^{i_{n-3}}+O(r^{n-2}))(u_\epsilon+\beta)^{2n \over n-4}d\mu_{\tilde g}+O(1).
\end{eqnarray*}
Since $\beta \leq C u_\epsilon$ in $B_{\tilde{\delta}}(x_0)$, then
\begin{eqnarray*}
&&\frac{1}{(n-3)!}\int_{B_{\tilde \delta}(x_0)}\nabla_{i_1\cdots i_{n-3}} f(x_0)x^{i_1}\cdots x^{i_{n-3}}(u_\epsilon+\beta)^{2n \over n-4} d\mu_{\tilde g}\\
&=&O(1)\int_{B_{\tilde \delta}(x_0)}|x|^{n-3}(u_\epsilon+\beta)^{2n \over n-4}d\mu_{\tilde g}\\
&=&O(1)\Big[\int_{M^n}|x|^{n-3} u_\epsilon^{2n \over n-4}d\mu_{\tilde g}+\frac{2n}{n-4}\beta\int_{M^n}|x|^{n-3} u_\epsilon^{n+4 \over n-4}d\mu_{\tilde g}\\
&&+\beta^2 \int_{M^n}O(|x|^{n-3} u_\epsilon^{8 \over n-4})d\mu_{\tilde g}+O(1)\Big].
\end{eqnarray*}
In particular, for dimension $n=5$, more precisely we have
\begin{eqnarray*}
&&\frac{1}{2}\int_{B_{\tilde \delta}(x_0)}\nabla_{ij} f(x_0)x^ix^j(u_\epsilon+\beta)^{2n \over n-4}d\mu_{\tilde g}\\
&=&\frac{1}{2n}\Delta f(x_0) \int_{B_{\tilde \delta}(x_0)}|x|^2(u_\epsilon+\beta)^{2n \over n-4}d\mu_{\tilde g}\\
&=&\frac{1}{2n}\Delta f(x_0)\Big[\int_{M^n}|x|^2 u_\epsilon^{2n \over n-4}d\mu_{\tilde g}+\frac{2n}{n-4}\beta\int_{M^n}|x|^2 u_\epsilon^{n+4 \over n-4}d\mu_{\tilde g}\\
&&+\beta^2 \int_{M^n}O(|x|^2 u_\epsilon^{8 \over n-4})d\mu_{\tilde g}+O(1)\Big].
\end{eqnarray*}
and the following asymptotics hold
\begin{eqnarray*}
\int_{M^n}|x|^{n-3} u_\epsilon^{2n \over n-4}d\mu_{\tilde g}&=&
O(\epsilon^{-3}) ;\\
\int_{M^n}|x|^{n-3} u_\epsilon^{n+4 \over n-4}d\mu_{\tilde g}
&=&\left \{\begin{array}{ll}
O(\epsilon^{-2}) &\hbox{~~if~~} n=5;\\
O(\epsilon^{-1}) &\hbox{~~if~~} n=6;\\
O(|\log \epsilon|) &\hbox{~~if~~} n=7;
\end{array}
\right.\\
\int_{M^n}|x|^{n-3} u_\epsilon^{8 \over n-4}d\mu_{\tilde g}&=&\left \{\begin{array}{ll}
O(\epsilon^{-1}) &\hbox{~~if~~} n=5;\\
O(1) &\hbox{~~if~~} n=6,7.
\end{array}
\right.
\end{eqnarray*}

Therefore, putting these facts above together, we obtain
\begin{eqnarray}
&&E_f[u_{0\epsilon}]\no\\
&=&\frac{\int_{M^n}\hat{u}_\epsilon P_{\tilde g}\hat{u}_\epsilon d\mu_{\tilde g}}{\Big(\int_{M^n}f\hat{u}_\epsilon^{2n \over n-4}d\mu_{\tilde g}\Big)^{n-4 \over n}}\no\\
&=& \frac{b_n \epsilon^4 \Big(\int_{M^n} u_\epsilon^{2n \over n-4}d\mu_{\tilde g}\Big)^{4 \over n}}{(\max_{M^n}f)^{n-4 \over n}} \frac{1+\beta(1+o_{\tilde \delta}(1))\tfrac{\int_{M^n} u_\epsilon^{n+4 \over n-4}d\mu_{\tilde g}}{\int_{M^n} u_\epsilon^{2n \over n-4}d\mu_{\tilde g}}+O_{\tilde \delta}(\epsilon^n)}{\Big(1+\tfrac{2n}{n-4}\beta\tfrac{\int_{M^n} u_\epsilon^{n+4 \over n-4}d\mu_{\tilde g}}{\int_{M^n} u_\epsilon^{2n \over n-4}d\mu_{\tilde g}}+\beta^2 \tfrac{\int_{M^n} O(u_\epsilon^{8 \over n-4})d\mu_{\tilde g}}{\int_{M^n} u_\epsilon^{2n \over n-4}d\mu_{\tilde g}}+O_{\tilde \delta}(\epsilon^{n-3})\Big)^{n-4 \over n}}\no\\
&=& \frac{q(S^n)}{(\max_{M^n}f)^{n-4 \over n}}\Big(1-\beta(1+o(1))\tfrac{\int_{M^n} u_\epsilon^{n+4 \over n-4}d\mu_{\tilde g}}{\int_{M^n} u_\epsilon^{2n \over n-4}d\mu_{\tilde g}}+O_{\tilde \delta}(\epsilon^{n-3})\Big).\label{initial_data_est_5_6_7}
\end{eqnarray}

(b) In dimension $n \geq 8$ and $(M^n,g_0)$ is locally conformally flat. From condition \eqref{lcf-higherdim}, near $x_0$, there holds
$$f(x)=f(x_0)+\tfrac{1}{(n-3)!}\nabla_{i_1 \cdots i_{n-3}} f(x_0)x^{i_1}\cdots x^{i_{n-3}}+O(r^{n-2}).$$
Together with the asymptotics
\begin{eqnarray*}
&& \int_{M^n}|x|^{n-3} u_\epsilon^{2n \over n-4}d\mu_{\tilde g}=O(\epsilon^{-3}),  \int_{M^n}|x|^{n-3} u_\epsilon^{n+4 \over n-4}d\mu_{\tilde g}=O(1),\\
&& \int_{M^n}|x|^{n-3} u_\epsilon^{8 \over n-4}d\mu_{\tilde g}=O(1),
\end{eqnarray*}
also a similar argument as in (a) yields
\begin{eqnarray}
&&E_f[u_{0\epsilon}]\no\\
&=&\frac{\int_{M^n}\hat{u}_\epsilon P_{\tilde g}\hat{u}_\epsilon d\mu_{\tilde g}}{\Big(\int_{M^n}f\hat{u}_\epsilon^{2n \over n-4}d\mu_{\tilde g}\Big)^{n-4 \over n}}\no\\
&=& \frac{b_n \epsilon^4 \Big(\int_{M^n} u_\epsilon^{2n \over n-4}d\mu_{\tilde g}\Big)^{4 \over n}}{(\max_{M^n}f)^{n-4 \over n}} \frac{1+\beta(1+o_{\tilde \delta}(1))\tfrac{\int_{M^n} u_\epsilon^{n+4 \over n-4}d\mu_{\tilde g}}{\int_{M^n} u_\epsilon^{2n \over n-4}d\mu_{\tilde g}}+O_{\tilde \delta}(\epsilon^n)}{\Big(1+\tfrac{2n}{n-4}\beta\tfrac{\int_{M^n} u_\epsilon^{n+4 \over n-4}d\mu_{\tilde g}}{\int_{M^n} u_\epsilon^{2n \over n-4}d\mu_{\tilde g}}+\beta^2 \tfrac{\int_{M^n} O(u_\epsilon^{8 \over n-4})d\mu_{\tilde g}}{\int_{M^n} u_\epsilon^{2n \over n-4}d\mu_{\tilde g}}+O_{\tilde \delta}(\epsilon^{n-3})\Big)^{n-4 \over n}}\no\\
&=& \frac{q(S^n)}{(\max_{M^n}f)^{n-4 \over n}}\Big(1-\beta(1+o(1))\tfrac{\int_{M^n} u_\epsilon^{n+4 \over n-4}d\mu_{\tilde g}}{\int_{M^n} u_\epsilon^{2n \over n-4}d\mu_{\tilde g}}+O_{\tilde \delta}(\epsilon^{n-3})\Big).\label{initial_data_est_lcf_hd}
\end{eqnarray}

In conclusion, for dimensions $n=5,6,7$ or $n \geq 8$ and $(M^n,g_0)$ is not conformally flat near $x_0$, we conclude from \eqref{initial_data_est_5_6_7} and \eqref{initial_data_est_lcf_hd} that
$$E_f[u_{0\epsilon}]=\frac{q(S^n)}{(\max_{M^n}f)^{n-4 \over n}}\big(1-O(\epsilon)\big).$$

The proof of the remained assertions is the same as in the proof of Lemma \ref{initial_hd_nlcf}.
\end{proof}

\section{Sequential convergence of the nonlocal $\mathbf{Q}$-curvature flow}\label{sect7}


\indent \indent In this section, we finish the proof of Theorem \ref{main_Thm} by showing the time sequential convergence of the nonlocal $Q$-curvature flow \eqref{gm_Q_flow_orig}-\eqref{constraint_factor} to a positive solution of $Q$-curvature equation \eqref{prescribed_Q-curvature}.\\

\noindent{\bf \underline {Proof of Theorem \ref{main_Thm}.}}

Under the assumptions of Theorem \ref{main_Thm}, using initial data which are constructed in Lemmas \ref{initial_hd_nlcf} and \ref{initial_ld_lcf}, in both (a) $n=5,6,7$ or $(M^n,g_0)$ is locally conformally flat, and (b) $n \geq 8$, Weyl tensor at $x_0$  is nonzero cases, we obtain
$$E_f[u_0] \leq \frac{q(S^n)}{(\max_{M^n}f)^{n-4 \over n}}-\epsilon_0,$$
for some small $\epsilon_0>0$. Then by Sobolev inequality on compact manifolds, for any given $\delta>0$ one has
\begin{eqnarray*}
&&\Big(\int_{M^n}u(t)^{2n \over n-4}d\mu_{g_0}\Big)^{n-4 \over n}\\
&\leq& (q(S^n)^{-1}+\delta)\int_{M^n}u(t)P_{g_0}u(t)d\mu_{g_0}+C(M^n,\delta)\int_{M^n}u(t)^2 d\mu_{g_0)}\\
&\leq&(q(S^n)^{-1}+\delta)E_f[u_0]\Big(\int_{M^n}fu(t)^{2n \over n-4}d\mu_{g_0}\Big)^{n-4 \over n}+C(M^n,\delta)\int_{M^n}u(t)^2 d\mu_{g_0}\\
&\leq&(q(S^n)^{-1}+\delta)(\max_{M^n}f)^{n-4 \over n}\Big[\tfrac{q(S^n)}{(\max_{M^n}f)^{n-4 \over n}}-\epsilon_0\Big]\Big(\int_{M^n}u(t)^{2n \over n-4}d\mu_{g_0}\Big)^{n-4 \over n}\\
&&+C(M^n,\delta)\int_{M^n}u(t)^2 d\mu_{g_0}
\end{eqnarray*}
By choosing $\delta=\frac{(\max_{M^n}f )^{n-4 \over n}}{2q(S^n)^2}\epsilon_0$, we obtain
\begin{equation}\label{lb_L^2}
\int_{M^n}u(t)^2 d\mu_{g_0}\geq C \Big(\int_{M^n} u(t)^{2n \over n-4}d\mu_g\Big)^{n-4 \over n} \geq C_0>0
\end{equation}
where we have used the uniform boundedness of the volume of the flow metric by Lemma \ref{bd_alpha}.

From Lemmas \ref{cons_energy}, \ref{bd_alpha_t}, \ref{asym_F_2} and Hardy-Littlewood-Sobolev inequality on compact manifolds, there exist sequences of $\{t_k\}$ with $t_k \to \infty$ as $k \to \infty$ and $\{u_k=u(t_k,\cdot)\}$, such that up to a subsequence as $k \to \infty$
\begin{eqnarray*}
&&\alpha(t_k) \to \alpha_\infty;\\
&&u_k \to u_\infty \hbox{~~weakly in ~~} H^2(M^n,g_0) \hbox{~~and strongly in~~} L^2(M^n,g_0);\\
&& -u_k+\tfrac{n-4}{2}\alpha_\infty P_{g_0}^{-1}(f u_k^{n+4 \over n-4}) \to 0 \hbox{~~strongly in~~} H^2(M^n,g_0).
\end{eqnarray*}
Thus, it also yields $u_\infty \geq 0$ and $u_\infty$ is a strong solution of
$$u_\infty=\tfrac{n-4}{2}\alpha_\infty P_{g_0}^{-1}(f u_\infty^{n+4 \over n-4}).$$
By the regularity theory of elliptic equations, one has
$$P_{g_0}u_\infty=\tfrac{n-4}{2}\alpha_\infty f u_\infty^{n+4 \over n-4}.$$
Thanks to the estimate \eqref{lb_L^2}, the strong maximum principle (cf. \cite{gur_mal} Theorem 2.2 or \cite{hy2} Proposition 3.1, respectively) yields $u_\infty>0$. Therefore, we conclude that up to a positive constant, the $Q$-curvature of a conformal metric $u_\infty^{4 \over n-4} g_0$ is equal to $f$. This completes the proof of the main Theorem \ref{main_Thm}.
\hfill $\Box$\\

\noindent {\bf Notes added after submission.} 

1. After a previous version of this article has been submitted to a journal, the author realized that some arguments of a recent paper \cite{hy2} are concerned with the left case mentioned in Remark \ref{leftcase} or Remark \ref{Rm_lcf_PMT}.

2. A former Phd student of Professor Xingwang Xu, Hong Zhang emailed me for his contribution of this work without attaching his article on 9 January 2015. Some more comments will be given by me only after I see his article of the work of nonlocal $Q$-curvature flow.


\begin{thebibliography}{99}
%
%

\bibitem{aubin}
T. Aubin, \textit{ $\acute{E}$quations diff$\acute{e}$rentielles non lin$\acute{e}$aires et probl$\grave{e}$me de Yamabe concernant la courbure scalaire}, J. Math. Pures Appl. (9) 55 (1976), no. 3, 269-296.

\bibitem{bfr}
P. Baird, A. Fardoun and R. Regbaoui, \textit{ $Q$-curvature flow on 4-manifolds}, Calc. Var. Partial Differential Equations 27 (2006), no. 1, 75-104.

\bibitem{branson}
T.P. Branson,\textit{ Group representations arising from Lorentz conformal geometry,} J. Funct. Anal. 74 (1987), 199-291.

\bibitem{bre3}
S. Brendle, \textit{ Convergence of the Yamabe flow for arbitrary initial energy,} J. Differential Geom. 69 (2005), no. 2, 217-278. 

\bibitem{bre1}
S. Brendle, \textit{Convergence of the Yamabe flow in dimension 6
and higher,} Invent. Math. 170 (2007), no. 3, 541--576.

\bibitem{bre2}
S. Brendle, \textit{Global existence and convergence for a higher
order flow in conformal geometry,} Annals of Math. 158 (2003),
323-343.

\bibitem{cy}
S.-Y. A. Chang and P. C.Yang,\textit{ Extremal metrics of zeta function determinants on 4-manifolds}, Annals of Math. 142 (1995), 171-212.

\bibitem{chen}
X. Chen, \textit{ The Paneitz-Sobolev constant of a closed Riemannian manifold and an application to the nonlocal $\mathbf{Q}$-curvature flow,} unpublished note (2014), arXiv: 1405.4412.

\bibitem{chxu}
X. Chen and X. Xu, \textit{The scalar curvature flow on
$\mathbf{S^n}$--- perturbation theorem revisited,} Invent. Math. 187 (2012), no. 2, 395-506.

\bibitem{dhl}
Z. Djadli, E. Hebey and M. Ledoux, \textit{ Paneitz-type operators and applications}, Duke Math. J. 104 (2000), no. 1, 129-169. 

\bibitem{es}
J. Escobar, R. Schoen, \textit{Conformal metrics with prescribed scalar curvature}, Invent. Math. 86 (1986), 243-254.

\bibitem{er}
P. Esposito and F. Robert, \textit{ Mountain pass critical points for Paneitz-Branson operators}, Calc. Var. Partial Differential Equations 15 (2002), no. 4, 493-517. 

\bibitem{feffgra}
C. Fefferman and C. R. Graham,\textit{ Conformal invariants,} in:
\'{E}lie Cartan et les Math\'{e}matiques d'aujourd'hui, Asterisque
(1985), 95-116.

\bibitem{gur_mal}
M. Gursky and A. Malchiodi, \textit{A strong maximum principle for the Paneitz operator and a non-local flow for the $Q$-curvature,} to appear in JEMS, arXiv:1401.3216.

\bibitem{hy1}
F. Hang and P. Yang, \textit{ Q curvature on a class of 3-manifolds,} to appear in CPAM.

\bibitem{hy2}
F. Hang and P. Yang, \textit{ Sign of Green's function of Paneitz operators and the Q-curvature,} to appear in IMRN, arXiv:1411.3924.

\bibitem{hy3}
F. Hang and P. Yang, \textit{ Q curvature on a class of manifolds with dimension at least five,} preprint (2014), arXiv:1411.3926.

\bibitem{hv}
E. Hebey and M.Vaugon, \textit{ Le probl$\acute{e}$me de Yamabe $\grave{e}$quivariant, Bulletin des Sciences Math$\grave{e}$matiques,} 117 (1993), 241-286.

\bibitem{hr}
E. Humbert and S. Raulot, \textit{ Positive mass theorem for the Paneitz-Branson operator}, Calc. Var. Partial Differential Equations 36 (2009), no.4, 525-531. 

\bibitem{jlx}
T. Jin, Y. Li and J. Xiong, \textit{ The Nirenberg problem and its generalizations: A unified approach,} preprint (2014), arXiv:1411.5743.

\bibitem{jx}
T. Jin and J. Xiong, \textit{ A fractional Yamabe flow and some applications,} J. Reine Angew. Math. 696 (2014), 187-223.

\bibitem{lp}
J. Lee and T. Parker, \textit{ The Yamabe problem}, Bull. Amer. Math. Soc. (N.S.) 17 (1987), no. 1, 37-91. 

\bibitem{mal_str}
A. Malchiodi and M. Struwe, \textit{$Q$-curvature flow on $S\sp 4$,}
J. Differential Geom. 73 (2006), no. 1, 1-44.

\bibitem{paneitz}
S. Paneitz, \textit{ A quadradic conformally covariant differential operator for arbitrary pseudo- Riemannian manifolds,} preprint, (1983).

\bibitem{sy}
R. Schoen and S.T. Yau, \textit{ On the proof of of the positive mass theorem conjecture in general relativity,} Comm. Math. Phys. 65 (1979), 45-76.

\bibitem{str}
M. Struwe,\textit{ A flow approach to Nirenberg problem,}Duke Math.
J.,128, no.1(2005),19-64.

\end{thebibliography}
\end{document}